\documentclass[10pt]{amsart}

\title{Preserve One, Preserve All}

\author{Meera Mainkar and Benjamin Schmidt}
\thanks{The second author learned about the Beckman-Quarles Theorem after the $d=2$ case was given to him as a puzzle during the Lie Group Actions in Riemannian Geometry held at Dartmouth College in 2017. He thanks Dmytro Yeroshkin for the excellent puzzle and Carolyn Gordon and Michael Jablonski for organizing the excellent conference.}

\newtheorem{thm}{Theorem}[section]
\newtheorem{lem}[thm]{Lemma}
\newtheorem{cor}[thm]{Corollary}
\newtheorem{prop}[thm]{Proposition}

\theoremstyle{definition}
\newtheorem{rem}{Remark}[section]
\newtheorem*{exmp*}{Examples}

\newtheorem*{defn*}{Definition}
\newtheorem*{thm*}{Theorem}	
\newtheorem*{conj*}{Conjecture}

\numberwithin{equation}{section}

\def\Pb{\ifmmode{\Bbb P}\else{$\Bbb P$}\fi}
\def\Z{\ifmmode{\Bbb Z}\else{$\Bbb Z$}\fi}
\def\Q{\ifmmode{\Bbb Q}\else{$\Bbb Q$}\fi}
\def\C{\ifmmode{\Bbb C}\else{$\Bbb C$}\fi}
\def\R{\ifmmode{\Bbb R}\else{$\Bbb R$}\fi}
\def\H{\ifmmode{\Bbb H}\else{$\Bbb H$}\fi}
\def\S{\ifmmode{S^2}\else{$S^2$}\fi}

\def\inj{\operatorname{inj}}

\def\diam{\operatorname{Diam}}

\def\S{\mathcal S}

\def\conv{\operatorname{conv}}

\usepackage{graphicx} 

\usepackage{amsmath,amsthm,amsfonts,amscd,flafter,epsf}
\usepackage{graphics}
\usepackage{epsfig}
\usepackage{psfrag}

\begin{document}

\begin{abstract}  Isometries of metric spaces $(X,d)$ preserve \textit{all} level sets of $d$.  We formulate and prove cases of a conjecture asserting if $X$ is a complete Riemannian manifold, then a function $f:X \rightarrow X$ preserving at least \textit{one} level set $d^{-1}(r)$, with $r>0$ small enough, is an isometry. \end{abstract}
\maketitle

\setcounter{secnumdepth}{1}

\setcounter{section}{0}

\section{\bf Introduction}
Given metric spaces $(X,d_X)$ and $(Y,d_Y)$ and a function $f:X \rightarrow Y$, let $$P_f=\{r>0\, \vert\, d_X^{-1}(r)\neq \emptyset\,\,\,\,\, \text{and}\,\,\,\,\,\, d_X(x,y)=r \implies d_Y(f(x),f(y))=r\}$$ $$SP_f=\{r>0\,\vert\,d_X^{-1}(r)\neq \emptyset\,\,\,\,\, \text{and}\,\,\,\,\, d_X(x,y)=r \iff d_Y(f(x),f(y))=r\}.$$  The classical Beckman-Quarles theorem asserts if $X=Y$ are Euclidean  $d$-space  $\mathbb{E}^d$ and $d \geq 2$, then $P_f=\emptyset$ or $f$ is an isometry \cite{BeQu}\footnote{The $d=2$ case reappeared as Problem 6 in the 1997 Brazilian Mathematics Olympiad.}.   The dimensional hypothesis is necessary. 
\vskip 5pt

\noindent{\textbf{Example 1:}}   The bijection $f$ of $\mathbb{E}^1$ that fixes irrational numbers and adds one to rational numbers satisfies $\mathbb{Q}_{>0} \subset SP_f.$\vskip 5pt

The Beckman-Quarles theorem does not generalize to Riemannian manifolds without additional assumptions.\vskip 5pt

\noindent{\textbf{Example 2:}} Given a subset $A$ of the unit sphere $S^n \subset \mathbb{E}^{n+1}$ with $A=-A$, the bijection $f$ of $S^n$ that fixes the complement of $A$ and is multiplication by $-1$ on $A$ satisfies $\{\frac{1}{2}\pi,\pi\} \subset SP_f$. \vskip5pt

The convexity radius of $S^n$ equals $\frac{1}{2}\pi$, motivating the following conjectural generalization.
\vskip 5pt

 \noindent{\textbf{Conjecture:}}  \textit{If $X$ is a complete Riemannian manifold with positive convexity radius $\conv(X)$ and $\dim(X)\geq 2,$ then for each function $f:X \rightarrow X$,  either $(0,\conv(X))\cap P_f =\emptyset$ or $f$ is an isometry.}\vskip 5pt

The conjecture holds for real hyperbolic spaces \cite{Ku} and unit spheres \cite{Ev}.  If $f$ is a \textit{bijection} of a locally compact geodesically complete CAT(0) space $X$ with path connected metric spheres, then $SP_f=\emptyset$ or $f$ is an isometry \cite{Be, An}; complete and simply connected Riemannian manifolds with nonpositive sectional curvatures are examples of such spaces.  Theorems A-C below provide additional evidence for the validity of the conjecture. \vskip 5pt

%\noindent{\textbf{Theorem A:}} If $X$ is isometric to a quaternionic hyperbolic space or the Cayley hyperbolic plane and $f:X \rightarrow X$ is a surjective function, then $SP_f=\emptyset$ or $f$ is an isometry. \vskip 5pt

%The proof of Theorem A proceeds as follows.  If $P_f \neq \emptyset$, then Lemma \ref{?} establishes that $f$ is a quasi-isometry of $X$ and is therefore finite distance from an isometry $g$ of $X$ by \cite{Pan}.  The function $h=g^{-1}\circ f$ is surjective, preserves a distance, and is finite distance from $\Id_X$.  Lemma \ref{} shows the image of a complete geodesic under a function having these three properties is a subset of a complete geodesic. It follows that $h=\Id_X$, or equivalently, $f=g$.  ADD BRIEF DISCUSSION ABOUT OTHER QI RIGID SPACES...  \vskip 5pt

\noindent{\textbf{Theorem A:}} \textit{Let $X$ be as in the conjecture. If a function $f:X\rightarrow X$ is surjective or continuous, and if there exist $\{r,R\}\subset (0,\conv(X))\cap SP_f$ with $r/R$ irrational, then $f$ is an isometry.}\vskip 5pt

A metric space $(X,d)$ is \textit{two-point homogenous} if the isometry group acts transitively on each level $d^{-1}(r)$; the connected two-point homogenous spaces consist of the Euclidean and rank one symmetric spaces \cite{Wa, Sz}.  The noncompact connected two-point homogenous spaces have infinite convexity radii and the compact connected two-point homogenous spaces have convexity radii equal to half their diameter. \vskip 5pt

\noindent{\textbf{Theorem B}:} \textit{Let $X$ be a connected two-point homogenous space with $\dim(X) \geq 2$ and $f:X\rightarrow X$ be a surjective or continuous function. If $(0,\frac{2}{3}\conv(X)) \cap SP_f \neq \emptyset$, then $f$ is an isometry.}\vskip 5pt

The proof of Theorem B does not use the classification of connected two-point homogenous spaces.  Instead, a unified approach is presented using the authors' Diameter Theorem in \cite{MaSc}.\vskip 5pt 

\noindent{\textbf{Theorem C}:} \textit{Let $X$ be as in the conjecture and have a periodic geodesic flow of period $1$.  If $f:X \rightarrow X$ is a surjective or continuous function, then $(0,\conv(X))\cap SP_f \subset \mathbb{Q}$ or $f$ is an isometry.} \vskip 5pt

Up to rescaling the metric, the positively curved (rank one) locally symmetric spaces satisfy the hypotheses of Theorem C.  Smooth spheres in each dimension are known to admit metrics as in Theorem C in addition to the constant curvature metrics \cite{Bes}.  

The proofs of Theorems A-C construct sequences of preserved distances converging to zero and then apply the following generalization of the Myers-Steenrod Theorem \cite{MySt} to conclude $f$ is an isometry.\vskip 5pt

\noindent{\textbf{Immersion Theorem:}} \textit{Let $X$ and $Y$ be Riemannian manifolds with $X$ complete and $\dim(X)\geq 2$. If $f:X \rightarrow Y$ is a function and $0$ is a limit point of $P_f$, then $f$ is a Riemannian immersion.}\vskip 5pt

In the Immersion Theorem, the assumption that $0$ is a limit point of $P_f$ cannot be weakened to the assumption, as in the conjecture, that $(0,\conv(X))\cap P_f \neq \emptyset$.\vskip 5pt

\noindent{\textbf{Example 3:}} The chromatic number of the plane is at most seven since there exists a function $c:\mathbb{E}^2 \rightarrow \{1,2,3,4,5,6,7\}$ with the property that for each $x, y \in \mathbb{E}^2$, if $d(x,y)=1$, then  $c(x) \neq c(y)$ \cite[Attributed to Iswell]{Ha}. Given vertices $\{v_1,v_2,v_3,v_4,v_5,v_6,v_7\}$ of a regular simplex in $\mathbb{E}^6$ with unit side lengths, define $f:\mathbb{E}^2 \rightarrow \mathbb{E}^6$ by $f(x)=v_{c(x)}$.  Then $1 \in (0,\conv(\mathbb{E}^2))\cap P_f $, but $f$ is discontinuous.   \vskip 5pt

Functions that are surjective or continuous and that strongly preserve a small distance are bijective (Lemmas \ref{bijective} and \ref{continuous}).  If a bijection preserves a distance then it also preserves the set of metric spheres having radii equal to that distance.  Smaller preserved distances are constructed by intersecting such spheres. The convexity hypothesis ensures nonempty intersections.  

For $X$ as in the conjecture, $x \in X$ and $r>0$, let  $S^x_r=\{y\,\vert\, d(x,y)=r\}$.  Let $|Y|$ denote the cardinality of a set $Y$.\vskip 5pt

\noindent \textbf{Sphere Intersections Theorem:} \textit{Let $x_1,x_2 \in X$ and $r_1, r_2 \in (0,\conv(X))$.} \begin{enumerate}
\item  $S^{x_1}_{r_1} \cap S^{x_2}_{r_2}\neq \emptyset  \iff |r_1-r_2| \leq d(x_1,x_2) \leq r_1+r_2,$
\item $|S^{x_1}_{r_1}\cap S^{x_2}_{r_2}|=1 \iff d(x_1,x_2)=|r_1-r_2|>0\,\,\, \text{or}\,\,\, d(x_1,x_2)=r_1+r_2.$\end{enumerate}
\vskip 5pt

The convexity radius is reviewed in section \ref{Intersection}, where also, the various implications in the Sphere Intersections Theorem are proved as independent lemmas.  Each one is proved assuming weaker hypotheses on the radii except for the implication $$S^{x_1}_{r_1} \cap S^{x_2}_{r_2}\neq \emptyset \impliedby  |r_1-r_2| \leq d(x_1,x_2) \leq r_1+r_2.$$  The importance of convexity in this implication is illustrated by the following example.\vskip 5pt

\noindent{\textbf{Example 4:}}  Let $x_1,x_2 \in S^2$ be a pair of antipodal points.  Then $d(x_1,x_2)=\pi=2\conv(S^2)$.  Given $r_1 \in (\frac{1}{2}\pi,\pi)$ and $r_2 \in (\pi-r_1,r_1)$ the intersection $S^{x_1}_{r_1} \cap S^{x_2}_{r_2}$ is empty while the inequalities $|r_1-r_2|<d(x_1,x_2)<r_1+r_2$ are valid.       \vskip 5pt

The Immersion Theorem is proved in section \ref{Immersion}.  Preliminary results about the structure of preserved distances are proved in section \ref{preserve} and Theorems A-C are proved in section \ref{Main}.

\section{\bf Sphere Intersections Theorem}\label{Intersection} In this section, $X$ denotes a complete Riemannian manifold.  The Riemannian structure induces a complete geodesic metric $$d:X\times X \rightarrow \mathbb{R}.$$ Given $x \in X$ and $r>0$, let $$S^x_r=\{y\, \vert\, d(x,y)=r\},\,\,\,\,\, B^x_r=\{y\,\vert\, d(x,y)<r\},\,\,\,\,\,D^x_r=\{y\, \vert\, d(x,y)\leq r\}.$$

A subset $Y \subset X$ is \textit{strongly convex} if for each $y_1, y_2 \in Y$, there is a unique minimizing geodesic in $X$ with endpoints $y_1$ and $y_2$, and moreover, this geodesic lies entirely in $Y$.  Sufficiently small metric balls are strongly convex \cite{Wh}.  The \textit{convexity radius} of $X$, denoted $\conv(X)$, is the supremum of positive numbers $r$ having the property that for each $x \in X$ and $0<s<r$, the open ball $B^x_s$ is strongly convex, provided such a positive number exists, and is zero otherwise.  

The \textit{injectivity radius} of a point $x \in X$, denoted $\inj(x)$, is the supremum of positive real numbers $r$ such that all geodesic segments of length $r$ issuing from $x$ are minimizing.  The injectivity radius of a point in $X$ depends continuously on the point.  The injectivity radius of $X$, denoted $\inj(X)$, equals the infimum of the injectivity radii of its points.  

\begin{lem}\label{convexity}
The inequality $\conv(X) \leq \frac{1}{2}\inj(X)$ holds.
%\item For each $x \in X$, the open ball $B^{x}_{\conv(X)}$ is strongly convex.
%\end{enumerate}
\end{lem}

\begin{proof}
The inequality follows easily from \cite{Kl, Di}, see e.g \cite[Lemma 3.3]{MaSc}.  %If $a,b \in B^{x}_{\conv(X)}$, let $m=\max\{d(x,a),d(x,b)\}$. For sufficiently small $\epsilon>0$, the open metric ball $B^{x}_{m+\epsilon}$ is strongly convex and contained in $B^{x}_{\conv(X)}$ from which (2) follows.
\end{proof}
The next lemma is well known; the proof is omitted.
\begin{lem}\label{stricttriangle}
Let $a,b,x \in X$.  If equality holds in the triangle inequality $$d(a,b)\leq d(a,x)+d(x,b),$$ then there is an arclength parameterized minimizing geodesic $\tau:[0,d(a,b)] \rightarrow X$ with $a=\tau(0)$, $x=\tau(d(a,x))$, and $b=\tau(d(a,b))$.
\end{lem}
%\begin{proof}
%Consider the path obtained by concatenating a unit speed minimizing geodesic joining $a$ to $x$ with a unit speed geodesic joining $x$ to $b$.  This path has length $d(a,x)+d(x,b)$.  As equality holds, these geodesics meet smoothly at $x$ as is easily seen by considering a strongly convex ball with center $x$.
%\end{proof}

% Let $x_1,x_2 \in X$ and $r_1,r_2 \in (0, \infty)$.

\begin{lem}\label{Intersectimpliesinequalities}
If $S^{x_1}_{r_1} \cap S^{x_2}_{r_2} \neq \emptyset$, then $|r_1-r_2|\leq d(x_1,x_2) \leq r_1+r_2$.
\end{lem}

\begin{proof}
Let $z \in S^{x_1}_{r_1} \cap S^{x_2}_{r_2}$.  The desired inequalities are derived by substituting the equalities $d(x_1,z)=r_1$ and $d(x_2,z)=r_2$ into the three triangle inequalities associated to the set $\{d(x_1,x_2),d(x_1,z),d(x_2,z)\}$.
 \end{proof}
 
 \begin{lem}\label{Inequalitiesimpliesintersect}
 If $\dim(X)\geq 2$, $|r_1-r_2|\leq d(x_1,x_2) \leq r_1+r_2$, and either \begin{enumerate}
 \item $r_1,r_2\in(0,\conv(X))$, or
 \item $0<r_2\leq \min\{r_1,\inj(x_2)\}$ and $r_1+2r_2 \leq \inj(x_1)$,
 \end{enumerate} then $$S^{x_1}_{r_1} \cap S^{x_2}_{r_2} \neq \emptyset.$$
 \end{lem}

\begin{proof}
By Lemma \ref{convexity}, the hypotheses imply $r_i\leq \inj(x_i)$ for $i=1,2$.  In particular, the spheres $S^{x_i}_{r_i}\neq \emptyset$ for $i=1,2$. If $d(x_1,x_2)=0$, then $r_1=r_2$, and $S^{x_1}_{r_1}=S^{x_2}_{r_2},$ concluding the proof in this case.  Now assume $d(x_1,x_2)>0$.  Without loss of generality, $r_2 \leq r_1$.  Set $$T_{-}=d(x_1,x_2)-r_2\,\,\,\,\, \text{and} \,\,\,\,\,\, T_{+}=d(x_1,x_2)+r_2.$$ The hypotheses imply the inequalities \begin{equation}\label{est1} |T_{-}|\leq r_1\end{equation}  and \begin{equation}\label{est2} r_1 \leq T_{+}\leq r_1+2r_2.\end{equation}  Let $\gamma:\mathbb{R} \rightarrow X$ be an arclength parameterized geodesic with $x_1=\gamma(0)$ and $x_2=\gamma(d(x_1,x_2))$.  Set $$a=\gamma(T_{-})\,\,\,\,\,\text{and}\,\,\,\,\,b=\gamma(T_{+}).$$  

As $r_2\leq \inj(x_2)$, the restrictions of the geodesic $\gamma$ to the length $r_2$ intervals $[T_{-},d(x_1,x_2)]$ and $[d(x_1,x_2), T_{+}]$ are minimizing.  Therefore \begin{equation}\label{sphere1} a,b \in S^{x_2}_{r_2}.\end{equation}    If $d(x_1,a)=r_1$ or $d(x_1,b)=r_1$, then $S^{x_1}_{r_1} \cap S^{x_2}_{r_2} \neq \emptyset,$ concluding the proof in these cases.  Now assume \begin{equation}\label{noteq} d(x_1,a)\neq r_1\,\,\,\,\,\, \text{and}\,\,\,\,\, d(x_1,b)\neq r_1.\end{equation}  

We now claim \begin{equation}\label{master}T_{-}\leq|T_{-}|=d(x_1,a)<r_1<T_{+}. \end{equation}  To verify this claim, note that by (\ref{est1}), $|T_{-}|\leq r_1\leq \inj(x_1)$, implying $$d(x_1, a)=d(\gamma(0),\gamma(T_{-}))=|T_{-}|\leq r_1,$$ and by (\ref{noteq}), $T_{-}\leq |T_{-}|=d(x_1,a)<r_1.$  Similarly, if the inequality $r_1\leq T_{+} $ in (\ref{est2}) is an equality, then $$d(x_1,b)=d(\gamma(0),\gamma(T_{+}))=d(\gamma(0),\gamma(r_1))=r_1,$$ contradicting (\ref{noteq}), and concluding the verification of (\ref{master}). 

% If $d(x_1,a)=r_1$ or $d(x_1,b)=r_1$, then $S^{x_1}_{r_1} \cap S^{x_2}_{r_2} \neq \emptyset,$ concluding the proof in these cases.  Now assume \begin{equation}\label{noteq} d(x_1,a)\neq r_1\,\,\,\,\,\, \text{and}\,\,\,\,\, d(x_1,b)\neq r_1.\end{equation}

We next claim \begin{equation}\label{master2} d(x_1,b)>r_1.\end{equation}  

To verify (\ref{master2}), first consider the case when hypothesis (2) holds.  In this case, $T_{+}=d(x_1,x_2)+r_2\leq r_1+2r_2\leq \inj(x_1),$ whence $$d(x_1,b)=d(\gamma(0),\gamma(T_+))=T_{+}=d(x_1,x_2)+r_2\geq (r_1-r_2)+r_2=r_1.$$  By (\ref{noteq}), the inequality is strict, concluding the verification of (\ref{master2}) in this case.  

To complete the verification of (\ref{master2}), now consider the case when hypothesis (1) holds.  If (\ref{master2}) fails, then $d(x_1,b)\leq r$, and by (\ref{noteq}), $d(x_1,b)<r_1$.  This inequality and (\ref{master}) imply that $a,b \in B^{x_1}_{r_1}$, a strongly convex ball since $r_1<\conv(X)$.  As $r_2<\conv(X)$, Lemma \ref{convexity} implies that the restriction of $\gamma$ to the length $2r_2$ interval $[T_{-},T_{+}]$ is a minimizing geodesic joining $a$ to $b$.  As $B^{x_1}_{r_1}$ is strongly convex, this minimizing geodesic is contained in $B^{x_1}_{r_1}$, or equivalently, \begin{equation}\label{contain} t \in [T_{-},T_{+}] \implies d(x_1,\gamma(t))<r_1.\end{equation} On the other hand, by (\ref{master})  there exists $\epsilon>0$ with   $$\epsilon< \min\{T_{+}-r_1,\conv(X)-r_1, \frac{1}{2}\inj(X)\}.$$    As $$r_1+\epsilon< \conv(X)+\epsilon \leq \frac{1}{2}\inj(X)+\epsilon <\inj(X),$$ the restriction of $\gamma$ to $[0,r_1+\epsilon]$ is a minimizing geodesic.  Therefore, $d(x_1,\gamma(r_1+\epsilon))=d(\gamma(0),\gamma(r_1+\epsilon))=r_1+\epsilon$, contrary to (\ref{contain}), concluding the verification of (\ref{master2}).

The inequalities (\ref{master}) and (\ref{master2}) imply that $S^{x_1}_{r_1} \cap S^{x_2}_{r_2} \neq \emptyset$ as will now be demonstrated.  As $\dim(X) \geq 2$ and $r_2 < \inj(x_2)$, the metric sphere $S^{x_2}_{r_2}$ is path connected. Let $\phi:[0,1] \rightarrow S^{x_2}_{r_2}$ be a continuous path with $\phi(0)=a$ and $\phi(1)=b$ and define $f:[0,1] \rightarrow \mathbb{R}$ by $f(t)=d(x_1,\phi(t))$.  Then $f(0)< r_1$ and $f(1)> r_1$ by (\ref{master}) and (\ref{master2}).  By the intermediate value theorem, there exists $t_0 \in (0,1)$ with $f(t_0)=r_1$.  It follows $$\phi(t_0) \in S^{x_1}_{r_1} \cap S^{x_2}_{r_2},$$ concluding the proof. 

%  By Lemma \ref{convexity} (SHOULD USE THAT $D^{x_2}_{r_2}$ Is Strongly Convex), the restriction of $\gamma$ to the interval $[T_{-},T_{+}]$ is minimizing and contains $p_1$ in its interior since $$T_{-}\leq r_1<r_1+\epsilon<T_{+}.$$  

%By Lemma \ref{convexity}, $D^{x_1}_{r_1}$ is strongly convex.  

  %It follows from (\ref{a}) and (\ref{c}) that $b=\gamma(T_{+}) \notin D^{x_1}_{r_1}$ or equivalently $d(x_1,b)>r_1$, establishing (\ref{b}) and concluding the proof of the Lemma.

%and set $p:=\gamma(r_1-\epsilon).$ As $$T_{+}-(r_1-\epsilon)=d(x,y)+r_2-r_1+\epsilon \leq (r_1+r_2)+r_2-r_1+\epsilon=2r_2+\epsilon<\inj(X),$$ the restriction of $\gamma$ to $[r_1-\epsilon,T_{+}]$ is a minimizing geodesic joining $p$ to $b$.

%As $p \in B^{x_1}_{r_1}$ and $\gamma(r_1+\epsilon) \notin B^{x_1}_{r_1}$, since $B^{x_1}_{r_1}$ is strongly convex, $b \notin B^{x_1}_{r_1}$, or equaivalently, \begin{equation} \label{far} d(x_1,b)>r_1. \end{equation}

\end{proof}

%\begin{rem}
%Lemmas \ref{Intersectimpliesinequalities}-\ref{Inequalitiesimpliesintersect} together imply statement (1) in the Intersecting Spheres Theorem.
%\end{rem}

%\begin{rem}\label{alternative}
%In Lemma \ref{Inequalitiesimpliesintersect}, the hypothesis $r_1,r_2 \in (0, \conv(X))$ may be replaced by the alternative hypotheses that $0<r_2\leq \min\{r_1,\inj(x_2)\}$ and $r_1+2r_2 \leq \inj(x_1)$.  Indeed, the proof presented above is easily modified under this alternative hypothesis as follows.  No modification is needed up to the derivation of (\ref{a}).  Then the alternative hypothesis implies $T_{+}=d(x_1,x_2)+r_2\leq r_1+2r_2\leq \inj(x_1)$ whence $$d(x_1,b)=d(\gamma(0),\gamma(T_+))=d(x_1,x_2)+r_2\geq (r_1-r_2)+r_2=r_1.$$ If equality holds, then $b$ lies in the intersection of the spheres. If the inequality is strict, then the last paragraph in the proof presented above applies without modification.
%\end{rem}

\begin{lem}\label{Pointimpliesequality}
If $\dim(X)\geq 2$,  $r_2< \inj(x_2)$, $r_1+r_2< \inj(x_1)$, $r_2\leq r_1$, and $|S^{x_1}_{r_1}\cap S^{x_2}_{r_2}|=1$, then $d(x_1,x_2)=r_1-r_2>0$ or $d(x_1,x_2)=r_1+r_2.$
\end{lem}

\begin{proof}
As $r_i< \inj(x_i)$, the metric spheres $S^{x_i}_{r_i}$ are embedded codimension one submanifolds of $X$.  

If $x_1 = x_2$ and $r_2<r_1$, then $S^{x_1}_{r_1}$ and $S^{x_2}_{r_2}$ have empty intersection.  If $x_1=x_2$ and $r_1=r_2$, then $S^{x_1}_{r_1} \cap S^{x_2}_{r_2}=S^{x_1}_{r_1}$ has dimension $\dim(X)-1>0$. Therefore $d(x_1,x_2)>0$.

Let $z$ be the unique point in $S^{x_1}_{r_1}\cap S^{x_2}_{r_2}$.  As $z$ is the unique point and $\dim(X) \geq 2$, the codimension one submanifolds $S^{x_1}_{r_1}$ and $S^{x_2}_{r_2}$ do \textit{not} intersect transversally at $z$.  Therefore $T_z S^{x_1}_{r_1}=T_zS^{x_2}_{r_2}$ as subspaces of $T_zX.$

For $i=1,2$, let $\gamma_i:[0,r_i]\rightarrow X$ be an arclength parameterized minimizing geodesic joining $x_i=\gamma_i(0)$ to $z=\gamma_i(r_i)$.  By Gauss' Lemma, $\dot{\gamma}_i(r_i)$ is perpendicular to the subspace $T_z S^{x_i}_{r_i}$ of $T_z X$.  Conclude $\dot{\gamma}_1(r_1)=\pm \dot{\gamma}_2(r_2)$.

If $\dot{\gamma}_1(r_1)=\dot{\gamma}_2(r_2):=v$ let $\gamma:\mathbb{R} \rightarrow X$ denote the complete geodesic in $X$ with $\dot{\gamma}(0)=-v$. Then $\gamma(r_2)=x_2$ and $\gamma(r_1)=x_1$. As $r_1<\inj(x_1)$ the geodesic $\tau:[0,r_1]\rightarrow X$ defined by $\tau(s)=\gamma(r_1-s)$ is unit speed and minimizing.  Therefore $$d(x_1,x_2)=d(\gamma(r_1),\gamma(r_2))=d(\tau(0),\tau(r_1-r_2))=r_1-r_2.$$  

If $\dot{\gamma}_1(r_1)=-\dot{\gamma}_2(r_2):=v$, then let $\gamma: \mathbb{R} \rightarrow X$ denote the complete geodesic with $\dot{\gamma}(0)=\dot{\gamma}_1(0)$.  Then $x_1=\gamma(0)$ and $x_2=\gamma(r_1+r_2)$, and since $r_1+r_2< \inj(x_1)$, $d(x_1,x_2)=d(\gamma(0),\gamma(r_1+r_2))=r_1+r_2.$
\end{proof}

\begin{lem}\label{Equalityimpliespoint1}
If $0<r_2<r_1<\inj(x_1)$ and $d(x_1,x_2)=r_1-r_2$, then $|S^{x_1}_{r_1}\cap S^{x_2}_{r_2}|=1$.
\end{lem}

\begin{proof}
As $r_1-r_2<r_1<\inj(x_1)$ there is a \textit{unique} arclength parameterized minimizing geodesic $\bar{\gamma}:[0,r_1-r_2]\rightarrow X$ joining $x_1=\bar{\gamma}(0)$ to $x_2=\bar{\gamma}(r_1-r_2)$.  Let $\gamma:\mathbb{R} \rightarrow X$ denote its complete extension.  Then $x_1=\gamma(0)$ and $x_2=\gamma(r_1-r_2)$.  Set $p=\gamma(r_1)$.  As $r_1<\inj(x_1)$, the restriction of $\gamma$ to $[0,r_1]$ is minimizing.  Therefore
$$d(x_1,p)=d(\gamma(0),\gamma(r_1))=r_1\,\,\,\,\, \text{and}\,\,\,\,\, d(x_2,p)=d(\gamma(r_1-r_2),\gamma(r_1))=r_2$$ and  
$p \in S^{x_1}_{r_1} \cap S^{x_2}_{r_2}.$  If $q \in S^{x_1}_{r_1} \cap S^{x_2}_{r_2}$, then $$r_1=d(x_1,q)\leq d(x_1,x_2)+d(x_2,q) =(r_1-r_2)+r_2=r_1.$$ By Lemma \ref{stricttriangle}, there is an arclength parameterized minimizing geodesic $\tau: [0,r_1] \rightarrow X$ joining $x_1=\tau(0)$ to $q=\tau(r_1)$ with $x_2=\tau(r_1-r_2)$.  By uniqueness, the restriction of $\tau$ to $[0,r_1-r_2]$ equals $\bar{\gamma}$, and consequently, the restriction of $\gamma$ to $[0,r_1]$ equals $\tau$.  Therefore $q=\tau(r_1)=\gamma(r_1)=p$.     
\end{proof}

%Given $x \in X$ and $r>0$, let $D^x_r=\{y\,\vert\, d(x,y)\leq r\}$ denote the closed metric ball with center $x$ and radius $r$.

\begin{lem}\label{Equalityimpliespoint2}
If  $r_1+r_2< \inj(x_1)$ and $d(x_1,x_2)=r_1+r_2$, then $|D^{x_1}_{r_1}\cap D^{x_2}_{r_2}|=1$ and $D^{x_1}_{r_1}\cap D^{x_2}_{r_2}=S^{x_1}_{r_1}\cap S^{x_2}_{r_2}$.
\end{lem}

\begin{proof}
As $r_1+r_2<\inj(x_1)$ there exists a \textit{unique} arclength parameterized minimizing geodesic $\gamma:[0,r_1+r_2]\rightarrow X$ joining $x_1=\gamma(0)$ to $x_2=\gamma(r_1+r_2)$.  Let $p=\gamma(r_1)$.  As $\gamma$ is arclength parameterized and minimizing, $$d(x_1,p)=d(\gamma(0),\gamma(r_1))=r_1\,\,\,\,\,\text{and}\,\,\,\,\, d(p,x_2)=d(\gamma(r_1),\gamma(r_1+r_2))=r_2$$ and $p \in S^{x_1}_{r_1}\cap S^{x_2}_{r_2}$.  If $q \in D^{x_1}_{r_1}\cap D^{x_2}_{r_2}$, then $$r_1+r_2 =d(x_1,x_2) \leq d(x_1,q)+d(q,x_2)\leq r_1+r_2.$$  By Lemma \ref{stricttriangle}, there is a minimizing unit speed geodesic $\tau:[0,r_1+r_2] \rightarrow X$ joining $x_1=\tau(0)$ to $x_2=\tau(r_1+r_2)$ with $q=\tau(r_1)$.  As $\gamma$ is unique, $\gamma$ equals $\tau$ and $p=\gamma(r_1)=\tau(r_1)=q$.
\end{proof}

\noindent \underline{\textit{Proof of Sphere Intersections Theorem.}}\vskip 5pt
Lemmas \ref{Intersectimpliesinequalities}-\ref{Inequalitiesimpliesintersect} together imply statement (1) in the Theorem.  Lemma \ref{convexity} and Lemmas \ref{Pointimpliesequality}-\ref{Equalityimpliespoint2} together imply statement (2) in the Theorem.
\qed \vskip 5pt
%We conclude this section with a lemma that will be used in Section \ref{Main}.

\section{\bf Immersion Theorem}\label{Immersion}%NOT HAPPY WITH THIS SECTION AS IT CURRENTLY STANDS

Let $(X,g)$ be a complete Riemannian manifold with $\dim(X)\geq 2$ and let $(Y,h)$ be a Riemannian manifold. Let $d_X$ and $d_Y$ denote the complete geodesic metrics on $X$ and $Y$ induced by the Riemannian metrics $g$ and $h$. Let $f:X \rightarrow Y$ be a function and assume that $0$ is a limit point of $P_f$.  

The Myers-Steenrod Theorem \cite{MySt}  asserts that a surjective distance preserving function between Riemannian manifolds is a smooth Riemannian isometry.  The Immersion Theorem-- that $f$ is a Riemannian immersion -- is a generalization of the Myers-Steenrod Theorem.  The proof here adapts Palais' proof \cite{Pa} of the Myers-Steenrod Theorem as presented in \cite{KoNo}.  

A preliminary well-known lemma concerns functions between inner product spaces of possibly unequal dimensions. 

%. Given $x \in X$ and $s>0$, let $S^x_s:=\{y \in X\, \vert\, d_X(x,y)=s\}$ and let $\bar{x}:=f(x)$ denote the metric $s$-sphere with center $x$ and the image point of $x$, respectively.  Given a curve $c:I \rightarrow X$ and $t\in I$, let $c_t:=c(t)$.

%\begin{lem}\label{lem: intersect}
%Assume that $x_1,x_2 \in X$ and that $r_1,r_2 \in \mathbb{R}$ satisfy $0<r_i<\inj(x_i)$ for $i=1,2$.  If $\vert r_1-r_2 \vert\leq d_X(x_1,x_2) \leq r_1+r_2,$ then $S(x_1,r_1)\cap S(x_2,r_2) \neq \emptyset.$
%\end{lem}

%\begin{proof}
%Add a Proof.
%Without loss of generality, $r_1 \geq r_2$.  The hypotheses imply that the metric spheres $S(x_i,r_i)$ are codimension one embedded spheres in $X$.  By Jordan separation[REFERENCE], $X \setminus S(x_1,r_1)$ has two connected components.  If these spheres do not intersect, then $S(x_2,r_2)$ is contained in one of these connected components.  Let $c:[0,\infty) \rightarrow X$ be a geodesic with $c(0)=x_1$ and $c(d_X(x_1,x_2))=x_2.$  

%If $S(x_2,r_2)$ is contained in the component of $X\setminus S(x_1,r_1)$ containing $x_1$, then since $c(d_X(x_1,x_2)+r_2) \in S(x_2,r_2) \subset S(x_1,r_1)$, $d_X(x_1,x_2)+r_2<r_1$.  If $S(x_2,r_2)$ is contained in the other component, then  $d_X(x_1,x_2)>r_2+r_1.$
%\end{proof}

\begin{lem}\label{linear}
Let $V_1$ and $V_2$ be real inner product spaces.  If a function $F:V_1 \rightarrow V_2$ satisfies $\langle u,w \rangle = \langle F(u),F(w) \rangle$ for all $u,w\in V_1$, then $F$ is a linear isometric map.
\end{lem}

\begin{proof}
It suffices to prove that $F$ is linear.  Let $u,w \in V_1$ and $\alpha \in \mathbb{R}$.  Use the hypothesis and bilinearity of the inner products to determine $$\langle F(\alpha u + w)-\alpha F(u) -F(w), F(\alpha u + w)-\alpha F(u) -F(w)\rangle=$$ $$\langle (\alpha u+w)-\alpha u-w, (\alpha u+w)-\alpha u-w \rangle=0.$$

%Let $\{v_1,\ldots,v_n\}$ be an orthonormal basis of $V$.  It suffices to prove that for all $(c_1,\ldots, c_n) \in \mathbb{R}^n$, \begin{equation}\label{linlin} F(c_1v_1+\cdots+c_n v_n)=c_1F(v_1)+\cdots+c_nF(v_n).\end{equation}

%The hypothesis implies that $\{F(v_1),\ldots,F(v_n)\}$ is an orthonormal set of vectors in $W$.  In particular, $\dim(V)\leq \dim(W)$.

%If $\dim(V)=\dim(W)$, then $\{F(v_1),\ldots,F(v_n)\}$ is an orthonormal basis of $W$. For each $i\in \{1,2,\ldots,n\}$, $$\langle F(c_1v_1+\cdots+c_n v_n),F(v_i)\rangle_2=\langle c_1 v_1+\cdots+ c_n v_n,v_i\rangle_1=c_i,$$ from which (\ref{linlin}) follows.

%If $k=\dim(W)-\dim(V)>0$, then complete $\{F(v_1),\ldots,F(v_n)\}$ to an orthonormal basis $\{F(v_1),\ldots, F(v_n),w_1,\ldots, w_k\}$ of $W$.  By the previous case, there exists $(b_1,\ldots,b_k) \in \mathbb{R}^k$ such that $$F(c_1v_1+\ldots+c_n v_n)=(c_1F(v_1)+\cdots +c_n F(v_n))+(b_1w_1+\cdots +b_k w_k).$$

%The hypothesis implies $$|| c_1v_1+\cdots+c_n v_n||_1^2=||F(c_1v_1+\cdots+c_n v_n)||_2^2.$$ Therefore $$c_1^2+\cdots+c_n^2=(c_1^2+\cdots+c_n^2)+(b_1^2+\cdots+b_k^2),$$ from which (\ref{linlin}) follows.

\end{proof}

%\begin{lem}\label{AI}
%If $l \in (0,\conv(X)) \cap P_f$, then for all $x,y \in X$, $d_Y(f(x),f(y)) \leq d_X(x,y)+2l$.
%\end{lem}

\begin{lem}\label{lem: lip}
The function $f$ is $1$-Lipschitz.
\end{lem}

\begin{proof}
Let $x,y \in X$ and $\epsilon>0$.  As $X$ is complete, there exists a minimizing geodesic $$\gamma:[0,d_X(x,y)] \rightarrow X$$  joining $x$ to $y$.  Set $\delta=\min\{\inj(\gamma(t))\,\vert\, t\in[0,d_X(x,y)]\}$ and choose $l \in P_f$ with $l<\min\{d_X(x,y),\delta/3,\epsilon/2\}.$  As $l<d_X(x,y)$, there exists $k \in \mathbb{N}$ such that $$k\cdot l<d_X(x,y)\leq (k+1)\cdot l.$$  For each integer $i$ with $0 \leq i \leq k$ set $x_i=\gamma(il)$.  The above inequalities imply that $d_X(x_k,y)<l$.  As $l< \delta/3$, Lemma \ref{Inequalitiesimpliesintersect}-(2) implies there exists $z\in S^{x_k}_l\cap S^{y}_l$.  As $\gamma$ is minimizing, $d_X(x_i,x_{i+1})=l$ for each integer $0\leq i \leq k-1.$ 

As $l \in P_f$ and $l< \epsilon/2$, the triangle inequality implies that $$d_Y(f(x),f(y)) \leq \sum_{i=0}^{k-1} d_Y(f(x_i),f(x_{i+1}))+d_Y(f(x_k),f(z))+d_Y(f(z),f(y))$$ $$=k\cdot l +2\cdot l<d_X(x,y)+\epsilon.$$
\end{proof}

\begin{lem}\label{lem: shortgeos}
For each $l\in P_f$ and arclength parameterized minimizing geodesic $\gamma:[0,l]\rightarrow X$, the curve $f\circ \gamma:[0,l] \rightarrow Y$ is a an arclength parameterized minimizing geodesic.
\end{lem}

\begin{proof}
By \cite[Proposition 3.8]{KoNo}, it suffices to prove that $f\circ \gamma$ is an isometric map of the interval $[0,l]$ into $Y$.  If $0\leq s_1<s_2\leq l$, then by Lemma \ref{lem: lip}, $d_Y(f(\gamma(s_1)),f(\gamma(s_2)))\leq s_2-s_1$.

The assumptions imply that $d_Y(f(\gamma(0)),f(\gamma(l)))=l$.  Therefore, if $0\leq t_1<t_2\leq l$, the triangle inequality implies $$l=d_Y(f(\gamma(0)),f(\gamma(l)))$$ $$\leq d_Y(f(\gamma(0)),f(\gamma(t_1)))+d_Y(f(\gamma(t_1)),f(\gamma(t_2)))+d_Y(f(\gamma(t_2)),f(\gamma(l)))$$ $$\leq (t_1-0)+(t_2-t_1)+(l-t_2)=l.$$ Conclude $d_Y(f(\gamma(t_1)),f(\gamma(t_2)))=t_2-t_1$.
\end{proof}

\noindent \underline{\textit{Proof of Immersion Theorem.}}\vskip 5pt

Fix $x \in X$.  Let $S_xX$ and $T_xX$ denote the unit tangent sphere and tangent space of $X$ at $x$, respectively.  Let $S_{f(x)}Y$ and $T_{f(x)}Y$ denote the unit tangent sphere and tangent space of $Y$ at $f(x)$, respectively. 

Choose $l \in P_f$ with $l<\min\{\inj(x),\inj(f(x))\}$.  Given $u \in S_xX$, denote by $\gamma_u:[0,l]\rightarrow X$ the arclength parameterized minimizing geodesic with $\dot{\gamma}_u(0)=u$ and let $\bar{\gamma}_u=f\circ \gamma_u$.  By  Lemma \ref{lem: shortgeos}, $\bar{\gamma}_u:[0,l]\rightarrow Y$ is an arclength parameterized minimizing geodesic in $Y$. Define $$F:S_xX \rightarrow S_{f(x)}Y$$ by $F(u)=\dot{\bar{\gamma}}_u(0)$ for each $u \in S_xX$.  This function extends to a function $$L:T_xX \rightarrow T_{f(x)}Y$$ defined by $L(\alpha u)=\alpha F(u)$ for each $\alpha \in \mathbb{R}$ and $u \in S_xX$.

Let $\exp_x$ and $\exp_{f(x)}$ denote the restrictions of the exponential maps of $X$ at $x$ and of $Y$ at $f(x)$ to the open balls $B^{0}_l\subset T_xX$ and $B^{0}_l\subset T_{f(x)} Y$, respectively.  By the choice of $l$, $\exp_x$ and $\exp_{f(x)}$ are diffeomorphisms onto the open balls $B^x_l\subset X$ and $B^{f(x)}_l\subset Y$.  Moreover, the restriction of $f$ to $B^{x}_l$ is given by \begin{equation}\label{ff} f=\exp_{f(x)}\circ L \circ \exp_{x}^{-1}.\end{equation} It suffices to prove that for each $u,w \in T_xX,$  \begin{equation}\label{key} g(u,w)=h(L(u),L(w)),\end{equation} as will now be explained.  If (\ref{key}) holds, then by Lemma \ref{linear}, $L$ is linear and isometric, and by (\ref{ff}), $f$ is smooth with derivative map at $x$ equal to $L$.  

It remains to establish the validity of (\ref{key}).  As $L$ satisfies $L(\alpha v)=\alpha L(v)$ for each $\alpha \in \mathbb{R}$ and $v \in V$ and carries unit vectors to unit vectors, it suffices to demonstrate (\ref{key}) for distinct unit vectors $u$ and $w$.  By Cauchy-Schwartz,  there exist $\theta$ and $\bar{\theta}$ such that $\cos(\theta)=g(u,w)$ and $\cos(\bar{\theta})=h(F(u),F(w))$.  

Let $\gamma_u$, $\gamma_w$, $\bar{\gamma}_u$, and $\bar{\gamma}_v$ be geodesic segments as defined above.  By the law of cosines (see e.g. \cite[Lemma, Page 170]{KoNo}), $$\cos(\theta)=\lim_{s \rightarrow 0} \frac{2s^2-d^2_X(\gamma_u(s),\gamma_w(s))}{2s^2} \,\,\,\,\,\text{and}\,\,\,\,\,\cos(\bar{\theta})=\lim_{s \rightarrow 0} \frac{2s^2-d^2_Y(\bar{\gamma}_u(s),\bar{\gamma}_w(s))}{2s^2}.$$  Therefore, it suffices to find a sequence $\{s_i\}$ of positive real numbers that converge to zero and satisfy $d_X(\gamma_u(s_i),\gamma_w(s_i))=d_Y(\bar{\gamma}_u(s_i),\bar{\gamma}_w(s_i)).$  

Define $h:[0,l] \rightarrow X$ by $h(s)=d(\gamma_u(s),\gamma_w(s))$.  Then $h$ is continuous and $h(0)=0$.  As $u$ and $w$ are distinct, there exists $\epsilon>0$ such that the restriction of $h$ to $[0,\epsilon]$ is a homeomorphism onto its image $[0,h(\epsilon)]$.  As $0$ is a limit point of $P_f$, the set $P_f \cap [0,h(\epsilon)]$ contains a sequence $\{t_i\}$ converging to zero.  Letting $s_i=h^{-1}(t_i)$, the sequence $\{s_i\}$ has the desired properties above.
\qed

\section{\bf Preserved Distances}\label{preserve}
In this section, $X$ denotes a complete Riemannian manifold with $\conv(X)>0$ and $\dim(X)\geq 2$.  Let $f:X \rightarrow X$ be a function.

\begin{lem}\label{same}
If $0<r<\conv(X)$ and $x,y \in X$ satisfy $S^{x}_r=S^{y}_r$, then $x=y$.
\end{lem}

\begin{proof}
Let $\gamma:[-r,r]\rightarrow X$ be an arclength parameterized geodesic with $\gamma(0)=x$.  By Lemma \ref{convexity}, $2r<\inj(X)$ so that $\gamma$ is the unique minimizing geodesic segment with endpoints $\gamma(-r)$ and $\gamma(r)$.  Therefore, $d(\gamma(\pm r), x)=r$ and $d(\gamma(-r),\gamma(r))=2r$.  The triangle inequality and the hypothesis $S^x_r=S^y_r$ imply $$2r=d(\gamma(-r),\gamma(r))\leq d(\gamma(-r),y)+d(y,\gamma(r))=r+r=2r.$$  By Lemma \ref{stricttriangle} there is a minimizing geodesic with endpoints $\gamma(-r)$ and $\gamma(r)$ and midpoint $y$.  As the segment $\gamma$ is unique, $x=y$.
\end{proof}

\begin{rem}
The convexity hypothesis in Lemma \ref{same} is necessary as illustrated by metric spheres in $S^2$ with antipodal centers and radii $\frac{1}{2}\pi$. 
\end{rem}

\begin{lem}\label{bijective}
If $(0,\conv(X))\cap SP_f \neq \emptyset$, then $f$ is injective.
\end{lem}

\begin{proof}
Let $r \in (0,\conv(X)) \cap SP_f$ and assume that $f(x)=f(y)$.  If $a\in S^x_r$, then since $r \in SP_f$, $$r=d(a,x)=d(f(a),f(x))=d(f(a),f(y))=d(a,y).$$  Conclude $S^{x}_r=S^{y}_r$ and by Lemma \ref{same}, $x=y$.
\end{proof}

\begin{lem}\label{continuous}
If $f$ is continuous and $(0,\conv(X))\cap SP_f \neq \emptyset$, then $f$ is surjective.
\end{lem}

\begin{proof}
Let $r \in (0,\conv(X))\cap SP_f$.  As $X$ is connected, it suffices to prove that the image of $f$ is both open and closed.  To achieve this, we demonstrate that if $p$ is in the image of $f$, then so too is the closed ball $D^{p}_{2r}$.

As a preliminary observation, note that if $x \in X$, then by Lemma \ref{bijective} and invariance of domain, the restriction of $f$ to $S^{x}_r$ is a homeomorphism onto $S^{f(x)}_r$.  

Now assume $p=f(a)$ and $d(p,q) \leq 2r$.  By Lemma \ref{Inequalitiesimpliesintersect}-(1), there exists $z \in S^{p}_r\cap S^q_r$. As the restriction of $f$ to $S^{a}_r$ is a homeomorphism onto $S^{p}_r$ and $z \in S^{p}_r$, there exists $b \in S^{a}_r$ with $f(b)=z$.  As the restriction of $f$ to $S^{b}_r$ is a homeomorphism onto $S^{z}_r$ and $q \in S^{z}_r$, there exists $c \in S^{b}_r$ with $f(c)=q$, completing the proof.  
\end{proof}

\begin{lem}\label{Preserve}
If $f$ is surjective, $x_1,x_2 \in X$, and $r_1,r_2 \in SP_f$,  then $$f(S^{x_1}_{r_1} \cap S^{x_2}_{r_2})=S^{f(x_1)}_{r_1}\cap S^{f(x_2)}_{r_2}.$$
\end{lem}

\begin{proof}
If $x \in S^{x_1}_{r_1}\cap S^{x_2}_{r_2}$, then $d(x,x_1)=r_1$ and $d(x,x_2)=r_2$.  As $r_1, r_2 \in P_f$, $d(f(x),f(x_1))=r_1$ and $d(f(x),f(x_2))=r_2$.  Therefore $$f(S^{x_1}_{r_1} \cap S^{x_2}_{r_2})\subset S^{f(x_1)}_{r_1}\cap S^{f(x_2)}_{r_2}.$$  %By Lemma \ref{bijective}, $f$ is bijective; repeating the argument with $f^{-1}$ demonstrates $$S^{f(x_1)}_{r_1}\cap S^{f(x_2)}_{r_2}\subset f(S^{x_1}_{r_1} \cap S^{x_2}_{r_2}).$$

If $y \in S^{f(x_1)}_{r_1} \cap S^{f(x_2)}_{r_2}$, then $d(y,f(x_1))=r_1$ and $d(y,f(x_2))=r_2$.  There exists $x$ such that $f(x)=y$.  As $r_1, r_2 \in SP_f$, $d(x,x_1)=r_1$ and $d(x,x_2)=r_2$.  Therefore $x \in S^{x_1}_{r_1} \cap S^{x_2}_{r_2}$ and $$S^{f(x_1)}_{r_1}\cap S^{f(x_2)}_{r_2}\subset f(S^{x_1}_{r_1} \cap S^{x_2}_{r_2}).$$
\end{proof}

\begin{lem}\label{Option}
Let $f$ be surjective and $r_1, r_2 \in SP_f$. If $r_1>r_2$, $r_1+r_2< \inj(X)$, and $d(x_1,x_2) \in \{r_1-r_2,r_1+r_2\}$, then $d(f(x_1),f(x_2)) \in \{r_1-r_2,r_1+r_2\}.$
\end{lem}

\begin{proof}
By Lemmas \ref{Equalityimpliespoint1} and \ref{Equalityimpliespoint2}, $|S^{x_1}_{r_1} \cap S^{x_2}_{r_2}|=1$.  Therefore, $|f(S^{x_1}_{r_1} \cap S^{x_2}_{r_2})|=1$.  By Lemma \ref{Preserve}, $|S^{f(x_1)}_{r_1}\cap S^{f(x_2)}_{r_2}|=1$. By Lemma \ref{Pointimpliesequality}, $$d(f(x_1),f(x_2)) \in \{r_1-r_2,r_1+r_2\}.$$
\end{proof}

\begin{lem}\label{Twice}
If $f$ is surjective and $r \in (0,\conv(X)) \cap SP_f$,  then $2r \in SP_f$.  %If $f$ is bijective, then $2r \in SP_f$.
\end{lem}

\begin{proof}
Assume $d(x_1,x_2)=2r$. By the Sphere Intersections Theorem-(2), $|S^{x_1}_{r} \cap S^{x_2}_{r}|=1$.  Therefore $|f(S^{x_1}_{r} \cap S^{x_2}_{r})|=1$. By Lemma \ref{Preserve}, $|S^{f(x_1)}_r \cap S^{f(x_2)}_r|=1$.  By the Sphere Intersection Theorem-(2), $d(f(x_1),f(x_2))=2r$.  Conclude $2r \in P_f$. By Lemma \ref{bijective}, $f$ is bijective; repeating the argument with $f^{-1}$ demonstrates $2r \in SP_f$
\end{proof}

\begin{lem}\label{Many}
Let $f$ be surjective and $r \in (0,\conv(X)) \cap SP_f$.  Let $k$ be the largest integer with the property that for each positive integer $j\leq k$, $jr \in SP_f$, provided a largest such integer exists, and let $k=\infty$ otherwise.  Then $kr\geq \conv(X)$.
\end{lem}

\begin{proof}
Note that by Lemma \ref{Twice}, $k \geq 2$.  We argue by contradiction.  Without loss of generality, $k< \infty$.  If $kr<\conv(X)$, then applying Lemma \ref{Option} to $f$ and $f^{-1}$ with $r_1=kr$ and $r_2=r$ implies that a pair of points $x_1,x_2 \in X$ satisfies $d(x_1,x_2) \in \{(k-1)r,(k+1)r\}$ if and only if $d(f(x_1), f(x_2))\in \{(k-1)r,(k+1)r\}$.  By the definition of $k$, $(k-1)r \in SP_f$.  It then follows $(k+1)r \in SP_f$, the desired contradiction.
\end{proof}

\begin{lem}\label{small}
Let $a,b \in X$ and $r \in (0,\conv(X))$.    \begin{enumerate} \item If $d(a,b)<r$, then $S^a_r \cap S^b_r\neq \emptyset$ and $S^a_{2r}\cap S^b_r=\emptyset$.  
\item If $S^a_r \cap S^b_r\neq \emptyset$, $S^a_{2r}\cap S^b_r=\emptyset$, and $r \in (0,\frac{2}{3}\conv(X))$, then $d(a,b)<r$. 
\end{enumerate}
\end{lem}

\begin{proof}
If $d(a,b)<r$, then since $r<\conv(X)$, Lemma \ref{Inequalitiesimpliesintersect} implies $S^{a}_r \cap S^{b}_r\neq \emptyset$. By Lemma \ref{Intersectimpliesinequalities},
 $S^{a}_{2r} \cap S^{b}_r=\emptyset$.
 
Next assume that $r<\frac{2}{3}\conv(X)$, $S^{a}_r \cap S^{b}_r\neq \emptyset$, and $S^{a}_{2r} \cap S^{b}_r=\emptyset$.  Since $S^{a}_r \cap S^{b}_r\neq \emptyset$, Lemma \ref{Intersectimpliesinequalities} implies $d(a,b)\leq 2r$.   By Lemma \ref{convexity}, $3r<\inj(X)$.  Therefore, since $S^{a}_{2r}\cap S^{b}_{r}=\emptyset$, Lemma \ref{Inequalitiesimpliesintersect} implies $d(a,b)>3r$ or $d(a,b)<r$.  Therefore $d(a,b)<r$.   
\end{proof}

\begin{rem}
The hypothesis in Lemma \ref{small}-(2) is likely not optimal.  If $X$ is the unit two sphere, then this statement is valid for $r\leq \frac{4}{5}\conv(X)$.
\end{rem}

\begin{lem}\label{Small}
If either 
\begin{enumerate}
\item $f$ is surjective and $r \in (0,\frac{2}{3}\conv(X))\cap SP_f$, or
\item $f$ is continuous and $r \in (0,\conv(X))\cap SP_f$,
\end{enumerate} then $d(x_1,x_2)<r$ if and only if $d(f(x_1),f(x_2))<r$.  In particular, \begin{enumerate}

\item For each $x \in X$, $f(D^x_r)=D^{f(x)}_r$, and

\item If $Y \subset X$ satifies $\diam(Y)=r$, then $\diam(Y)=\diam(f(Y))$

\end{enumerate}
\end{lem}

\begin{proof}
Assertions (1) and (2) in the Lemma follow immediately from the main assertion of the Lemma.

We first prove the main assertion assuming hypothesis (1). By Lemma \ref{Twice}, $2r \in SP_f$.  By Lemma \ref{Preserve}, $f(S^{x_1}_{r} \cap S^{x_2}_r)=S^{f(x_1)}_{r} \cap S^{f(x_2)}_r$ and $f(S^{x_1}_{2r} \cap S^{x_2}_r)=S^{f(x_1)}_{2r} \cap S^{f(x_2)}_r$. The main assertion of the Lemma is now a consequence of Lemma \ref{small}. %Assertions (1) and (2) follow immediately.

We conclude with the proof of the main assertion assuming hypothesis (2).  By invariance of domain and Lemmas \ref{bijective} and \ref{continuous}, $f$ is a homeomorphism.  It follows that if $x \in X$, then the function $h:B^x_r \rightarrow \mathbb{R}$ defined by $h(y)=d(f(x),f(y))$ is an interval in $[0,r)\cup(r,\infty)$.  The conclusion follows since $h(x)=0$.
\end{proof}

%\begin{lem}\label{Small2}
%If $\dim(X)\geq 2$, $f$ is continuous, and $r \in (0,\conv(X)) \cap SP_f$, then $d(x_1,x_2)<r$ if and only if $d(f(x_1),f(x_2))<r$.
%\end{lem}

%\begin{proof}

%\end{proof}

\begin{lem}\label{noexpand}
If $r \in (0,\conv(X)) \cap P_f$ and $d(x_1,x_2)\leq 2r$, then $d(f(x_1),f(x_2)) \leq 2r$.
\end{lem}

\begin{proof}
By the Sphere Intersection Theorem-(1), there exists $z \in S^{x_1}_{r}\cap S^{x_2}_r.$ As $r \in P_f$, $d(f(x_1),f(x_2))\leq d(f(x_1),f(z))+d(f(z),f(x_2))=2r$.
\end{proof}

\begin{lem}\label{oy}
If $f$ is surjective, $r_1,r_2 \in (0,\conv(X))\cap SP_f$, and $r_1-r_2 \leq 2r_2<r_1+r_2$, then $r_1-r_2 \in SP_f$
\end{lem}

\begin{proof}
Assume that $d(a,b)=r_1-r_2$.  By Lemma \ref{Option}, $d(f(a),f(b))=r_1-r_2$ or $d(f(a),f(b))=r_1+r_2$.  As $r_1-r_2\leq 2r_2$, Lemma \ref{noexpand} implies $d(f(a),f(b))\leq 2r_2<r_1+r_2$, whence $d(f(a),f(b))=r_1-r_2$.  By Lemma \ref{bijective}, $f$ is bijective; repeating the argument with $f^{-1}$ demonstrates $r_1-r_2 \in SP_f$.
\end{proof}

Given $x \in \mathbb{R}$, let $\lfloor x \rfloor \in \mathbb{Z}$ denote the largest integer less than or equal to $x$.

\begin{prop}\label{Difference}
If $f$ is surjective, $r_1,r_2 \in (0,\conv(X))\cap SP_f$, and $r_1>r_2$, then $r_1-\lfloor r_1/r_2 \rfloor r_2 \in SP_f\cup \{0\}$. \end{prop}

\begin{proof}
Note that $\lfloor r_1/r_2 \rfloor r_2\leq r_1<(1+\lfloor r_1/r_2 \rfloor)r_2.$ The conclusion holds trivially when the first inequality is an equality.  Now consider the case when $\lfloor r_1/r_2 \rfloor r_2< r_1<(1+\lfloor r_1/r_2 \rfloor)r_2.$  If $\lfloor r_1/r_2 \rfloor=1$, then $r_1-r_2<r_1<2r_2$, and by Lemma \ref{oy}, $r_1-r_2 \in SP_f.$

Now assume that $\lfloor r_1/r_2 \rfloor \geq 2.$ Then $$r_1-\lfloor r_1/r_2 \rfloor r_2<r_2\leq r_1/2<\conv(X)/2.$$  If $d(a,b)=r_1-\lfloor r_1/r_2 \rfloor r_2$, then applying Lemma \ref{Small} with $r=r_2$ implies $$d(f(a), f(b))<r_2.$$  As $\lfloor r_1/r_2\rfloor r_2< r_1<\conv(X)$, Lemma \ref{Many} implies $\lfloor r_1/r_2\rfloor r_2 \in SP_f$.  It then follows from Lemma \ref{Option}, applied to the radii $r_1$ and $\lfloor r_1/r_2 \rfloor r_2$, that $d(f(a),f(b))= r_1-\lfloor r_1/r_2 \rfloor r_2$. By Lemma \ref{bijective}, $f$ is bijective; repeating the argument with $f^{-1}$ demonstrates  $r_1-\lfloor r_1/r_2 \rfloor r_2 \in SP_f.$
\end{proof}

\section{\bf Theorems A-C}\label{Main}

Theorem A is based on the following lemma.

\begin{lem}\label{LimitPoint}
Let $S$ be a subset of $(0, \infty)$ satisfying:
\begin{enumerate}
\item If $a,b \in S$ and $a>b$, then $a-\lfloor a/b \rfloor b \in S \cup \{0\}$.
\item There exist $a,b \in S$ with $a/b$ irrational.
\end{enumerate}
Then $0$ is a limit point of $S$.
\end{lem}

\begin{proof}
Let $\epsilon>0$. We will show $S \cap(0,\epsilon)\neq \emptyset$.  To this end, consider a strictly decreasing sequence $\{s_i\}$ in $S$ constructed as follows:  Let $a, b \in S$ be as in (2) with $a>b$.  Set $s_1=a$, $s_2=b$.  Define $s_3=s_1-\lfloor s_1/s_2 \rfloor s_2.$  Verify $s_2>s_3>0$ and $s_2/s_3$ is irrational.  Defining $s_i=s_{i-2}-\lfloor s_{i-2}/s_{i-1} \rfloor s_{i-1}$ iteratively produces the desired sequence.  As $S$ is bounded below, the strictly decreasing sequence $\{s_i\}$ is Cauchy.  Therefore, for $n$ sufficiently large $$s_{n+1}=s_{n-1}-\lfloor s_{n-1}/s_n \rfloor s_n \leq s_{n-1}-s_{n}<\epsilon.$$ 
\end{proof}

\noindent \underline{\textit{Proof of Theorem A.}}\vskip 5pt 
By Lemmas \ref{bijective} and \ref{continuous}, $f$ is a bijection.  Let $S=(0,\conv(X))\cap SP_f$.  The set $S$ satisfies Lemma \ref{LimitPoint}-(1) by Proposition \ref{Difference} and Lemma \ref{LimitPoint}-(2) by hypothesis. Therefore, zero is a limit point of $S$.  The Immersion Theorem implies that $f$ is a Riemannian immersion.    Bijective Riemannian immersions are isometries, concluding the proof. \qed \vskip 5pt

%\noindent{\textit{Proof Two of Theorem A Outline.}} 

%The next Lemma is an immediate consequence of Lemma \ref{Small}.  Same as above, but different approach to showing $r_1-r_2$ is preserved.  Have lemma showing points at distance $r_1-r_2$ map to points at distance $r_1-r_2$ or to points at distance $r_1+r_2$ .  Choose $x,y \in S(p,r_1+r_2)$ with $d(x,y)=r_2$.  Then $d(f(p),f(x))$, $d(f(p),f(y))$ in $\{r_1+r_2,r_1-r_2\}$ and $d(f(x),f(y))=2r_2$.  If $f(x) \in S(f(p),r_1-r_2)$ then $f(y)$    

%\begin{lem}\label{Diameter}
%Let $f$ be a bijection and $Y\subset X$ a compact subset.  If $$\diam(Y) \in (0,\frac{2}{3}\conv(X))\cap SP_f,$$ then $$\diam(Y)=\diam(f(Y)).$$
%\end{lem} 

%\begin{proof}

%\end{proof}

%\vskip 5pt

%\begin{proof}
%Let $a, b \in f(Y)$, $x=f^{-1}(a)$, and $y=f^{-1}(b)$.  If $d(x,y)=\diam(Y)$ then $d(a,b)=\diam

%For each $y_1,y_2 \in Y$, $d(y_1,y_2) \leq \diam(Y)$; equality holds for some pair $(y_1,y_2) \in Y \times Y$ since $Y$ is compact.  Then   

%(\ref{diameter}), then $$d(f(y_1),f(y_2))=\diam(Y)$$ since $\diam(Y) \in SP_f$.  If equality fails in (\ref{diameter}), then $d(f(y_1),f(y_2))<\diam(Y)$ by Lemma \ref{Small}.  Therefore $$\diam(f(Y))\leq \diam(Y).$$ \end{proof}

Theorem $B$ is based on the following specialization of the main theorem in \cite{MaSc}. \vskip 5pt

\noindent{\bf Diameter Theorem:} \textit{If $0<r<\conv(X)$ and if $\gamma:[0,2r]\rightarrow X$ is an arclength parameterized geodesic, then the function $$g(t)=\diam(D^{\gamma(0)}_r \cap D^{\gamma(t)}_r)$$ is continuous, monotonically decreasing, and satisfies $g(t)>2r-t$ for $t \in (0,2r)$.} \vskip 5pt

Given a pair of points $x$ and $y$ in the Euclidean plane and $r>0$, the intersection $D^x_r \cap D^y_r$ has diameter  $r$ if and only if $d(x,y)=\sqrt{3}r$.  The next Corollary is a generalization of this fact for connected two-point homogenous spaces.

\begin{cor}\label{corrr}
If $X$ is a connected two point homogenous space and $0<r<\conv(X)$, then there is a unique $\bar{r} \in (0,2r)$ with the property that for all $x,y \in X$ satisfying $d(x,y)\leq 2r$, $$\diam(D^x_r \cap D^y_r)=r \iff d(x,y)=\bar{r}.$$  Moreover, $\bar{r} \in (r,2r)$.
\end{cor}

\begin{proof}
Fix a geodesic as in the Diameter Theorem and let $g:[0,2r]\rightarrow \mathbb{R}$ be the associated diameter function.  As $X$ is two-point homogeneous, it suffices to prove that there is a unique $\bar{r} \in (0,2r)$ such that $g(\bar{r})=r$, and moreover, $\bar{r} \in (r,2r)$.  By the Diameter Theorem, $g(t)$ is continuous, monotonically decreasing, and satisfies $g(r)>2r-r=r$.  By Lemma \ref{Equalityimpliespoint2}, $g(2r)=0$.  The conclusion follows.  
\end{proof}

\begin{lem}\label{bar}
If $X$ is a connected two-point homogenous space, $f:X \rightarrow X$ is a bijection, $r \in (0,\frac{2}{3}\conv(X))\cap SP_f$, and $\bar{r}\in (r,2r)$ is as in Corollary \ref{corrr}, then $\bar{r} \in SP_f$.
\end{lem}

\begin{proof}
By Corollary \ref{corrr}, if $d(x,y)=\bar{r}$, then $\diam(D^x_r \cap D^y_r)=r.$ By Lemma \ref{Small} $\diam(D^{f(x)}_r\cap D^{f(y)}_r)=\diam(f(D^x_r \cap D^{y}_r))=r$.  In particular, $D^{f(x)}_r\cap D^{f(y)}_r$ is nonempty.  If $z \in D^{f(x)}_r\cap D^{f(y)}_r$, then $d(f(x),f(y))\leq d(f(x),z)+d(z,f(y))\leq 2r.$  By Corollary \ref{corrr}, $d(f(x),f(y))=\bar{r}$ and $\bar{r} \in P_f$.  Repeating this argument after replacing $f$ with $f^{-1}$ demonstrates $\bar{r} \in SP_f$.
\end{proof}

\noindent \underline{\textit{Proof of Theorem B.}}\vskip 5pt
By Lemmas \ref{bijective} and \ref{continuous}, $f$ is a bijection.  Define $l_0:=r$ and let $\bar{l}_0=\bar{r} \in (l_0,2l_0)$ be as in Corollary \ref{corrr}.  Define $l_1=\bar{l}_0-l_0$.  Then $0<l_1<l_0$.

By Lemma \ref{convexity}, $l_0+\bar{l}_0<3r<2\conv(X)\leq \inj(X).$ Apply Lemma \ref{oy} with $r_1=\bar{l}_0$ and $r_2=l_0$ to conclude $l_1 \in SP_f$.  

For $i \geq 2$, define $l_i$ inductively by $l_i:=\bar{l}_{i-1}-l_{i-1}$.  Repeating the above argument, the sequence $\{l_i\}$ is strictly decreasing and satisfies $l_i \in SP_f.$  As the sequence $\{l_i\}$ is bounded below by $0$, it is Cauchy.  Therefore, given $\epsilon>0$, for $i$ sufficiently large, $l_i-\lfloor l_i/l_{i+1}\rfloor l_{i+1} <\epsilon$. By Proposition \ref{Difference}, $l_i-\lfloor l_i/l_{i+1}\rfloor l_{i+1}  \in SP_f$.  Therefore $0$ is a limit point of $SP_f$.  By the Immersion Theorem, $f$ is a bijective Riemannian immersion, hence an isometry.  
\qed \vskip 5pt

Theorem C is based on the following well known density lemma. \vskip 5pt

\begin{lem}\label{density}
If $r \in (0, \infty)$ is irrational, then the set $\{nr - \lfloor nr \rfloor\, \vert\, n \in \mathbb{N}\}$ is dense in $[0,1]$.
\end{lem}

\noindent \underline{\textit{Proof of Theorem C.}}\vskip 5pt
By Lemmas \ref{bijective} and \ref{continuous}, $f$ is a bijection.  Assume that $r \in (0,\conv(X))\cap SP_f$ is irrational.  Let $\epsilon >0$.  By the Immersion Theorem, it suffices to prove $(0,\epsilon)\cap SP_f\neq \emptyset.$  By Lemma \ref{density}, there exists $n \in \mathbb{N}$ such that $0<nr-\lfloor nr \rfloor<\epsilon$ and $nr-\lfloor nr \rfloor<\inj(X)$.  We claim that $nr-\lfloor nr \rfloor \in SP_f$.

Given $x,y \in X$ with $d(x,y)=nr-\lfloor nr \rfloor$, let $\gamma:\mathbb{R} \rightarrow X$ be the arclength parameterized geodesic with $\gamma(0)=x$ and $\gamma(nr-\lfloor nr \rfloor)=y$.  As $\gamma$ is periodic with period one, $\gamma(nr)=\gamma(nr-\lfloor nr \rfloor)=y$.  

%Consider the set $\{\gamma(ir)\, \vert\, i \in \mathbb{N}\}$.  
Let $\tilde{\gamma}:\mathbb{R} \rightarrow X$ be the arclength parameterized geodesic with $\tilde{\gamma}(0)=f(\gamma(0))$ and $\tilde{\gamma}(r)=f(\gamma(r))$.  We claim that for all $i \in \mathbb{N}$,

\begin{equation}\label{star} f(\gamma(ir))=\tilde{\gamma}(ir). \end{equation}  The case $i=1$ in (\ref{star}) holds by construction; the remaining cases $i>1$ will be established using strong induction.  If (\ref{star}) holds for all $0\leq k<i$, then since $$d(\gamma((i-2)r),\gamma((i-1)r))=r=d(\gamma((i-1)r),\gamma(ir))$$ and $$d(\gamma((i-2)r),\gamma(ir))=2r,$$ 
Lemmas \ref{stricttriangle} and \ref{Twice} imply that $f(\gamma((i-2)r))$, $f(\gamma((i-1)r))$, and $f(\gamma(ir)$ lie in a common minimizing geodeisc segment of length $2r$.  Since $f(\gamma(kr))=\tilde{\gamma}(kr)$ when $k=(i-2)$ and $k=(i-1)$, this minimizing geodesic segment is the restriction of $\tilde{\gamma}$ to the interval $[(i-2)r,ir]$, verifying (\ref{star}) when $k=i$.

As $\tilde{\gamma}$ is periodic with period one, $f(y)=f(\gamma(nr))=\tilde{\gamma}(nr)=\tilde{\gamma}(nr-\lfloor nr\rfloor)$.  Therefore, the restriction of $\tilde{\gamma}$ to the interval $[0,nr-\lfloor nr\rfloor]$ is a geodesic segment of length $nr-\lfloor nr \rfloor$ joining $f(x)$ to $f(y)$.  As $nr-\lfloor nr \rfloor<\inj(X)$, $d(f(x),f(y))=nr-\lfloor nr \rfloor$ and $nr-\lfloor nr \rfloor \in P_f$.  Repeating the argument with $f^{-1}$ demonstrates $nr-\lfloor nr \rfloor \in SP_f$, concluding the proof.
\qed

\end{document}